\documentclass[a4paper,12pt,twoside]{amsart}
\usepackage[english]{babel}
\usepackage{amssymb, amsmath, eucal, enumerate}
\usepackage[all]{xy}
\usepackage[left=3cm,right=3cm,top=3cm,bottom=3cm]{geometry}

\usepackage[svgnames,hyperref]{xcolor}
\newcommand{\mycolor}{Navy}
\usepackage[pdfauthor={Hoang-Son Do and  Giang Le}, colorlinks, linktocpage, citecolor = \mycolor, linkcolor = \mycolor, urlcolor = \mycolor]{hyperref}
\newtheorem{The}{Theorem}[section]
\newtheorem{Lem}[The]{Lemma}
\newtheorem{Prop}[The]{Proposition}
\newtheorem{Cor}[The]{Corollary}

\newcommand{\C}{\mathbb{C}}

\newcommand{\Z}{\mathbb{Z}}

\newcommand{\M}{\mathcal{M}}

\begin{document}
 \title[Conditional plurisubharmonic envelopes]
 {On the conditional plurisubharmonic envelopes of bounded functions} 
\setcounter{tocdepth}{1}
\author{Hoang-Son Do} 
\address{Institute of Mathematics \\ Vietnam Academy of Science and Technology \\18
Hoang Quoc Viet \\Hanoi \\Vietnam}
\email{hoangson.do.vn@gmail.com, dhson@math.ac.vn}
\author{Giang Le}
\address{Department of Mathematics, Hanoi National University of Education, 136-Xuan Thuy, Cau Giay, Hanoi, VietNam}
\email{legiang@hnue.edu.vn, legiang01@yahoo.com}
\date{\today\\ The second author was supported by the Vietnam National Foundation for Science and Technology Development (NAFOSTED) under grant number 101.04-2018.03. }
\maketitle
\begin{abstract}
In this paper, we extend some recent results of Guedj-Lu-Zeriahi \cite{GLZ19} about
 psh envelopes of bounded functions on bounded domains in $\C^n$. 
 We also present a result on the regularity of psh envelopes.
\end{abstract}
\section{Introduction}
In \cite{GLZ19}, Guedj-Lu-Zeriahi studied quasi-plurisubharmonic envelopes on compact K\"ahler manifolds and
plurisubharmonic envelopes on domains of $\C^n$. By using and extending an approximation process due to Berman
 \cite{Ber19}, they show that the (quasi-)plurisubharmonic
 envelope of a viscosity super-solution is a pluripotential super-solution of a given complex 
 Monge-Amp\`ere equation. Our goal is to extend Guedj-Lu-Zeriahi's results for {\it conditional 
 plurisubharmonic envelopes} on domains of $\C^n$.

Let $\Omega\subset \C^n$ be a bounded domain. 
Denote by $\M$ the set of Borel measures $\mu$ on $\Omega$ satisfying $\mu=(dd^c\varphi)^n$ for some bounded
plurisubharmonic function $\varphi$ in $\Omega$.
If $\mu\in\M$ and $u$ is a bounded function in $\Omega$ then 
 we define
\begin{center}
	$P(u, \mu, \Omega):=(\sup\{v\in PSH(\Omega)\cap L^{\infty}(\Omega):v\leq u, (dd^cv)^n\geq \mu \})^*$.
\end{center}
By \cite{Kol98}, we have $fd\lambda\in\M$ for every $f\in L^p(\Omega), p>1$, where $\lambda$ is the
Lebesgue measure in $\C^n$.
If $f\in L^p(\Omega), p>1,$  then we also denote $P(u, f, \Omega):=P(u, fd\lambda, \Omega)$.
The first main result of this paper is the following:
  \begin{The}\label{main1}
  	Assume that $\Omega\subset\C^n$ is a bounded pseudoconvex domain.
  	Suppose that $f\in L^{p}(\Omega) (p>1)$ and $g\in C(\Omega)$ are non-negative functions.
  	 If $u$ is a bounded viscosity 
  	subsolution to the equation 
  	\begin{equation}
  	(dd^cw)^n=gd\lambda,
  	\end{equation}
  	on $\Omega$ then $(dd^cP(u, f, \Omega))^n\leq \max\{f, g\}d\lambda$ in the pluripotential sense in $\Omega$.
  \end{The}
\begin{Cor}
	Assume that $\Omega\subset\C^n$ is a bounded pseudoconvex domain and
	and $0\leq f, g\in L^p(\Omega), p>1$. Suppose that  $u$ is a continuous plurisubharmonic 
	on $\Omega$ such that $(dd^cu)^n=gd\lambda$ in the pluripotential sense. Then
	\begin{center}
		$(dd^cP(u, f, \Omega))^n\leq \max\{f, g \}d\lambda.$
	\end{center}
\end{Cor}
In this paper, we also study the continuity of $P(u, f, \Omega)$ when $u$ is continuous. 
Our second main result is the following:
\begin{The}\label{main0}
	Assume that $\Omega$ is a smooth strictly pseudoconvex domain. If $0\leq f\in L^p(\Omega), p>1$, and
	$u\in C(\overline{\Omega})$ then $P(u, f, \Omega)\in C(\overline{\Omega})$.
\end{The}
\begin{Cor}
Assume that $\Omega\subset\C^n$ is a smooth strictly pseudoconvex domain and $U\subset\Omega$
is a hyperconvex domain.  Then, for every $E\Subset U$, for each $0\leq f\in L^p(\Omega), p>1,$
if $P(-\chi_E, f, U)$ is continuous then $P(-\chi_E, f, \Omega)$ is continuous.
\end{Cor}
\section{Some general properties}
In this section, we give some properties of $P(u, \mu, \Omega)$, mainly about the convergence and stability.
Some of them have been proved in \cite{GLZ19} for the case $\mu=0$.
\begin{Prop}\label{prop1}
	Let $u$ be a bounded function on $\Omega$ and $\mu\in\M.$ Denote 
 \begin{center}
 	$T= \{v\in PSH(\Omega)\cap L^{\infty}(\Omega):v\leq u$ quasi everywhere,
 	$(dd^cv)^n\geq\mu \}$.
 \end{center}
Then $P(u, \mu, \Omega)\in T$. Moreover, $P(u, \mu, \Omega)=\sup\{v: v\in T \}$.
\end{Prop}
Here {\it $v\leq u$ quasi everywhere } means that there exists a pluripolar set $N$ such that $v\leq u$
on $\Omega\setminus N$.
\begin{proof}
	Since negligible sets are pluripolar \cite{BT82}, we have 
	$P(u, \mu, \Omega)= \sup_{v\in S}v$ quasi everywhere, where
	\begin{center}
		$S=\{v\in PSH(\Omega)\cap L^{\infty}(\Omega):v\leq u, (dd^cv)^n\geq \mu \}.$
	\end{center}
	 Hence, $P(u, f, \Omega)\leq u$ quasi everywhere. 
	 
	 By Choquet lemma, there exists a sequence of functions $u_j\in S$ such that 
	 $P(u, f, \Omega)=(\sup_ju_j)^*$. Note that if $v,w \in PSH(\Omega)\cap L^{\infty}(\Omega)$
	 and $(dd^cv)^n, (dd^cw)^n\geq \mu$ then $(dd^c\max\{v, w \})^n\geq \mu$. Hence 
	 $(dd^c(\max_{j\leq k}u_j))^n\geq\mu$ for every $k$. Letting $k\rightarrow\infty$ and using
	 \cite[Theorem 2.6]{BT82}(see also \cite[Theorem 3.6.1]{Kli91}), we get
	 $(dd^cP(u, f, \Omega))^n\geq \mu$. Then $P(u, f, \Omega)\in T$.
	 
	 Now, let $v$ be an arbitrary element of $T$. Then there exists $\varphi\in PSH^-(\Omega)$ such that
	 $\{v>u\}\subset\{\varphi=-\infty\}$. Denote $M=\sup |u-v|$. We have
	  $$v_{\epsilon}:=v+\max\{\epsilon\varphi,-M \}\in S,$$
	  for every $\epsilon>0$.  Letting $\epsilon\searrow 0$, we obtain
	  \begin{center}
	  	$v=(\lim\limits_{\epsilon\to 0^+}v_{\epsilon})^*\leq P(u, \mu, \Omega)$.
	  \end{center}
  Thus $P(u, \mu, \Omega)=\sup\{v: v\in T \}$.
\end{proof}
\begin{Cor}
		Let $u$ be a bounded function on $\Omega$ and $\mu\in\M.$ Then
		\begin{center}
		$P(u, \mu, \Omega)=P(P(u, 0, \Omega), \mu, \Omega).$	
		\end{center}
\end{Cor}
\begin{Prop}\label{prop2}
	Let $u$ be a bounded function on $\Omega$ and $\mu\in\M.$  If $\Omega_j$ is an increasing sequence of relative compact domains in $\Omega$ such that $\cup_j\Omega_j=\Omega$ then
	$P(u, \mu, \Omega_j)$ decreases to $P(u, \mu, \Omega)$.
\end{Prop}
\begin{proof}
	By the definition, we have 
	\begin{center}
		$P(u, \mu, \Omega)\leq P(u, \mu, \Omega_{j+1})\leq P(u, \mu, \Omega_j),$
	\end{center}
on $\Omega_j$ for every $j$. Denote $v=\lim\limits_{j\to\infty}P(u, \mu, \Omega_j)$. Then
$v$ is a bounded plurisubharmonic function on $\Omega$ satisfying
\begin{equation}\label{eq1 prop2}
P(u, \mu, \Omega)\leq v,
\end{equation}
and
\begin{equation}\label{eq2 prop2}
(dd^cv)^n\geq \mu.
\end{equation}
It follows from Proposition \ref{prop1}, 
that $P(u, \mu, \Omega_j)\leq u$ quasi everywhere on $\Omega_j$. Then $v\leq u$ quasi everywhere
on $\Omega$. Hence, by the last assertion of Proposition \ref{prop1} and by \eqref{eq2 prop2}, we get
\begin{equation}\label{eq3 prop2}
v\leq P(u, \mu, \Omega).
\end{equation}
Combining \eqref{eq1 prop2} and \eqref{eq3 prop2}, we obtain $v=P(u, \mu, \Omega)$.
 Thus $P(u, \mu, \Omega_j)$ decreases to $P(u, \mu, \Omega)$ as $j\rightarrow\infty$.
\end{proof}
\begin{Prop}\label{prop3}
	Let $u, u_j (j\in\Z^+)$ be bounded functions on $\Omega$ and $\mu\in\M.$  Then
	the following statements hold:
	\item [(i)] If $u_j$ decreases to  $u$ as $j\rightarrow\infty$ then $P(u_j, \mu, \Omega)$
	decreases to $P(u, \mu, \Omega)$.
	\item [(ii)] Assume that $u_j$ is continuous for every $j$. 
	If $u_j$ increases to  $u$ as $j\rightarrow\infty$ then $P(u_j, \mu, \Omega)$
	increases to $P(u, \mu, \Omega)$ almost everywhere.
\end{Prop}
\begin{proof}
(i) By the definition, we have
\begin{center}
	$P(u, \mu, \Omega)\leq P(u_{j+1}, \mu, \Omega)\leq P(u_j, \mu, \Omega),$
\end{center} 
for every $j$. Then
\begin{equation}\label{eq1 prop3}
v:=\lim\limits_{j\to\infty}P(u_j, \mu, \Omega)\geq P(u, \mu, \Omega).
\end{equation}
Since $(dd^cP(u_j, \mu, \Omega))^n\geq \mu$ for every $j$, we also have
\begin{equation}\label{eq2 prop3}
(dd^cv)^n\geq \mu.
\end{equation}
It follows from Proposition \ref{prop1} that
 $P(u_j, \mu, \Omega)\leq u_j$ quasi everywhere on $\Omega_j$. 
Letting $j\rightarrow\infty$, we get $v\leq u$ quasi everywhere
on $\Omega$. Hence, by the last assertion of Proposition \ref{prop1} and by \eqref{eq2 prop2}, we have
\begin{equation}\label{eq3 prop3}
v\leq P(u, \mu, \Omega).
\end{equation}
Combining \eqref{eq1 prop3} and \eqref{eq3 prop3}, we obtain $v=P(u, \mu, \Omega)$.
Thus $P(u_j, \mu, \Omega)$ decreases to $P(u, \mu, \Omega)$ as $j\rightarrow\infty$.\\

(ii) By the defintion, we have
\begin{center}
	$P(u, \mu, \Omega)\geq P(u_{j+1}, \mu, \Omega)\geq P(u_j, \mu, \Omega),$
\end{center} 
for every $j$. Then
\begin{equation}\label{eq4 prop3}
v:=(\lim\limits_{j\to\infty}P(u_j, \mu, \Omega))^*\leq P(u, \mu, \Omega).
\end{equation}
We will show that $v=\sup_{w\in T}w$, where
 \begin{center}
	$T= \{w\in PSH(\Omega)\cap L^{\infty}(\Omega):w\leq u$ quasi everywhere,
	$(dd^cw)^n\geq \mu \}$.
\end{center}
Since $(dd^cP(u_j, \mu, \Omega))^n\geq \mu$ for every $j$, we  have
\begin{equation}\label{eq5 prop3}
(dd^cv)^n\geq \mu.
\end{equation}
Combining \eqref{eq4 prop3} and \eqref{eq5 prop3} and using Proposition \ref{prop1}, we get that
\begin{equation}\label{eq6 prop3}
v\in T.
\end{equation}
Let $\varphi\in T$. 
 Since $\varphi-u\leq 0$ and $u_j-u\nearrow 0$, we have $\max\{\varphi-u_j, 0 \}$ decreases to $0$.
Denote by $\hat{\varphi}$ 
the upper semicontinuous extension of $\varphi$ to $\overline{\Omega}$, i.e.,
\begin{center}
$\hat{\varphi}(\xi):= \lim_{r\to 0^+} \sup_{B(\xi,r)\cap \Omega} \varphi, \ \forall\xi \in \partial\Omega.$
\end{center}
Then $\max\{\hat{\varphi}-u_j, 0 \}$ decreases to $0$ on $\overline{\Omega}$ as $j\rightarrow\infty$.
It follows from Dini's theorem that $\max\{\varphi-u_j, 0 \}$
converges uniformly on $\overline{\Omega}$ to $0$. Hence, for every $\epsilon>0$, there exists $j$ such that
\begin{center}
$\varphi-\epsilon\leq u_j.$
\end{center}
Then
\begin{center}
	$\varphi-\epsilon\leq P(u_j, \mu, \Omega)\leq v.$
\end{center}
Since $\varphi$ and $\epsilon$ are arbitrary, we get  
\begin{equation}\label{eq7 prop3}
v\geq\sup_{w\in T}w.
\end{equation} 
Combining \eqref{eq6 prop3} and \eqref{eq7 prop3}, we have
\begin{center}
	$v=\sup_{w\in T}w.$
\end{center}
Hence, by Proposition \ref{prop1}, we obtain $v=P(u, \mu, \Omega)$.
Thus, $P(u_j, \mu, \Omega)$
increases to $P(u, \mu, \Omega)$ almost everywhere.
\end{proof} 
\begin{Prop}\label{prop4}
Let $u$ be a bounded function on $\Omega$ and $0\leq f, g\in L^p(\Omega)$ for some $p>1$.
Then, there exists
a uniform constant $C>0$ such that
\begin{center}
	$|P(u, f, \Omega)-P(u, g, \Omega)|\leq C(\|f-g\|_{L^p(\Omega)})^{1/n}.$
\end{center}
\end{Prop}
\begin{proof}
	Let $D$ be a smooth strictly pseudoconvex domain in $\C^n$ such that $\Omega\Subset D$. Then, by
	\cite[Corollary 3.1.3]{Kol98}, there exists $\phi\in PSH(D)\cap C(\overline{D})$ such that
	$(dd^c\phi)^n=\chi_{\Omega}|f-g|$ and $\phi|_{\partial D}=0$.
	By Proposition \ref{prop1}, we have
	\begin{center}
		$P(u, f, \Omega)+\phi|_{\Omega}\leq P(u, g, \Omega),$
	\end{center}
and
\begin{center}
	$P(u, g, \Omega)+\phi|_{\Omega}\leq P(u, f, \Omega).$
\end{center}
Then
	\begin{equation}\label{eq1 prop4}
	\sup\limits_{\Omega}|P(u, f, \Omega)-P(u, g, \Omega)|\leq \sup\limits_{\Omega}|\phi|.
	\end{equation}
	 Using \cite[Theorem 1.1]{GKZ} for
	$\phi/(\|f-g\|_{L^p(\Omega)})^{1/n}$ and $0$, with $\gamma=0$, we have
	\begin{equation}\label{eq2 prop4}
	\sup\limits_{D}\dfrac{|\phi|}{(\|f-g\|_{L^p(\Omega)})^{1/n}}\leq C,
	\end{equation}
	where $C>0$ is a uniform constant.

		Combining \eqref{eq1 prop4}, \eqref{eq2 prop4}, we get
		\begin{center}
			$|P(u, f, \Omega)-P(u, g, \Omega)|\leq C(\|f-g\|_{L^p(\Omega)})^{1/n}.$
		\end{center}
\end{proof}
\begin{Prop}\label{prop5}
	Let $\Omega\subset\C^n$ be a bounded hyperconvex domain. Assume that
	 $u\in USC(\Omega)\cap L^{\infty}(\Omega)$,
	$v\in  LSC(\Omega)\cap L^{\infty}(\Omega)$, $\mu\in\M$ and $W\Subset\Omega$.
	Denote $M=\sup\limits_{W}|u-v|$. Then
	\begin{center}
		$Cap(\{|P(u, \mu, W)-P(v, \mu, W)|\geq M\epsilon \}, \Omega)
		\leq\dfrac{2(n!)^2}{\epsilon^n}Cap(\{|u-v|\geq\epsilon \}\cap\overline{W}, \Omega),$
	\end{center}
for every $\epsilon>0$.
\end{Prop}
Here $Cap(E, \Omega)$ is the relative capacity defined by Bedford-Taylor \cite{BT82} as follows:
\begin{equation}\label{eq def cap}
	Cap(E, \Omega)=\sup\{\int\limits_E(dd^cv)^n: v\in PSH(\Omega, [0, 1]) \}.
\end{equation}
\begin{proof}
	Denote $E_1=\{u-v\geq\epsilon\}\cap\overline{W}, E_2=\{u-v<-\epsilon \}\cap W$ and 
	$\delta=Cap(\{|u-v|\geq\epsilon \}\cap\overline{W}, \Omega)$. Then
	\begin{center}
	$Cap(E_j, \Omega)\leq \delta,\quad j=1, 2.$
	\end{center}
Since $E_1$ is compact and $E_2$ is open, we have 
\begin{center}
	$Cap(E_j, \Omega)=Cap^*(E_j, \Omega),\quad j=1, 2.$
\end{center}
Denote $E=E_1\cup E_2$. We have
\begin{equation}
Cap^*(E, \Omega)\leq Cap^*(E_1, \Omega)+Cap^*(E_2, \Omega)\leq 2\delta.
\end{equation}
Let $h_E=\sup\{h\in PSH^-(\Omega): h|_E\leq -1 \}$. It follows from \cite[Proposition 6.5]{BT82} that
\begin{equation}
\int\limits_{\Omega}(dd^ch_E^*)^n=Cap^*(E, \Omega)\leq 2\delta.
\end{equation}
By using \cite[Lemma 1]{Xin} for $h_E^*$ and $0$, we get
\begin{center}
	$\int\limits_{\Omega}(-h_E^*)^n(dd^ch)^n\leq (n!)^2\int\limits_{\Omega}(dd^ch_E^*)^n\leq 2(n!)\delta,$
\end{center}
for all $h\in PSH(\Omega, [0, 1])$. Hence
\begin{equation}
Cap(\{h_E^*<-\epsilon \}, \Omega)\leq\dfrac{2(n!)^2\delta}{\epsilon^n}.
\end{equation}
Note that, by \cite{BT82}, $h_E^*=h_E$ quasi everywhere. Then, by the definition of $h_E$, we have
\begin{center}
	$u+M h_E^*\leq v$ and $v+M h_E^*\leq u$,
\end{center}
quasi everywhere in $W$. Hence
\begin{center}
	$P(u, \mu, W)+M h_E^*\leq P(v, \mu, W)$ and 	$P(v, \mu, W)+M h_E^*\leq P(u, \mu, W)$.
\end{center}
Then 
\begin{center}
	$Cap(\{|P(u, \mu, W)-P(v, \mu, W)|\geq M\epsilon \}, \Omega)\leq 
	Cap(\{h_E^*<-\epsilon \}, \Omega)\leq\dfrac{2(n!)^2\delta}{\epsilon^n}.$
\end{center}
\end{proof}
\section{Proof of the main theorems}
\subsection{Proof of Theorem \ref{main1}}
We use the same method as in the proof of \cite[Theorem 3.9.]{GLZ19}. First, we prove a special case of
Theorem \ref{main1}.
\begin{Prop}\label{prop1 main1}
	Assume that $\Omega\subset\C^n$ is a bounded smooth strictly pseudoconvex domain and
	 $0\leq f, g\in C(\overline{\Omega})$.
If $u\in C(\overline{\Omega})$ is a viscosity 
subsolution to the equation 
\begin{equation}
(dd^cw)^n=gd\lambda,
\end{equation}
on $\Omega$ then $(dd^cP(u, f, \Omega))^n\leq \max\{f, g\}d\lambda$ in the pluripotential sense in $\Omega$.
\end{Prop}
\begin{proof}
	By \cite{BT76}, for every $j\in\Z^+$, there exists $u_j\in PSH(\Omega)\cap C(\overline{\Omega})$ such that
	\begin{equation}\label{eq uj prop1 main1}
	(dd^cu_j)^n=\max\{e^{j(u_j-u)}g, f \}d\lambda,
	\end{equation}
	in the pluripotential sense in $\Omega$ and $u_j=u$ on $\partial\Omega$.
	
	By \cite{EGZ}, $u_j$ satisfies \eqref{eq uj prop1 main1} in the viscosity sense. Applying the viscosity comparison principle \cite{EGZ, DDT} to the equation $(dd^cw)^n=\max\{e^{j(w-u)}g, f \}d\lambda$, we get
	$u_j\leq u$ and $u_j\leq u_{j+1}$ for every $j\in\Z^+$. Denote $v:=(\lim\limits_{j\to\infty} u_j)^*$.
	 We have
	\begin{center}
$\max{\{f, g\}}d\lambda\geq (dd^cu_j)^n\stackrel{weak}{\longrightarrow}(dd^cv)^n\geq fd\lambda,$
	\end{center}
	and
	\begin{center}
	$v\leq P(u, f, \Omega)\leq u.$
	\end{center}
It remains to show that $v\geq P(u, f, \Omega)$.
For every $j\in\Z^+$, $\epsilon>0$ and $h\in PSH(\Omega, [-1, 0))$, we denote
\begin{center}
	$E(j, \epsilon, h)=\{z\in\Omega: u_j<P(u, f, \Omega)-\epsilon+\epsilon h \}$.
\end{center}
By Proposition \ref{prop1}, we have $(dd^c P(u, f, \Omega))^n\geq fd\lambda$. Then
\begin{equation}\label{eq1 prop1 main1}
\int\limits_{E(j, \epsilon, h)}(fd\lambda+\epsilon^n(dd^ch)^n)\leq 
\int\limits_{E(j, \epsilon, h)}(dd^c(P(u, f, \Omega)+\epsilon dd^ch))^n.
\end{equation}
By the Bedford-Taylor comparison principle, we have
\begin{equation}\label{eq2 prop1 main1}
\int\limits_{E(j, \epsilon, h)}(dd^c(P(u, f, \Omega)+\epsilon dd^ch))^n
\leq \int\limits_{E(j, \epsilon, h)} (dd^cu_j)^n.
\end{equation}
Since $u_j-u\leq -\epsilon$ in $E(j, \epsilon, h)$, we get
\begin{equation}\label{eq3 prop1 main1}
\int\limits_{E(j, \epsilon, h)} (dd^cu_j)^n=\int\limits_{E(j, \epsilon, h)} \max\{e^{j(u_j-u)}g, f\}d\lambda
\leq \int\limits_{E(j, \epsilon, h)} \max\{e^{-j\epsilon}g, f\}d\lambda.
\end{equation}
Combining \eqref{eq1 prop1 main1}, \eqref{eq2 prop1 main1} and \eqref{eq3 prop1 main1}, we obtain
\begin{center}
	$\epsilon^n\int\limits_{E(j, \epsilon, h)}(dd^ch)^n\leq 
	\int\limits_{E(j, \epsilon, h)} (\max\{e^{-j\epsilon}g, f\}-f)d\lambda.$
\end{center}
Then
\begin{center}
	$\epsilon^n\int\limits_{\{u_j\leq P(u, f, \Omega)-2\epsilon\}}(dd^ch)^n
	\leq \int\limits_{\Omega} (\max\{e^{-j\epsilon}g, f\}-f)d\lambda.$
\end{center}
Since $h\in PSH(\Omega, [-1, 0))$ is arbitrary, it implies that
\begin{center}
	$Cap(\{u_j\leq P(u, f, \Omega)-2\epsilon\}, \Omega)\leq 
	\dfrac{1}{\epsilon^n}\int\limits_{\Omega} (\max\{e^{-j\epsilon}g, f\}-f)d\lambda,$
\end{center}
where $Cap(. , \Omega)$ is the relative capacity of Bedford-Taylor (see \eqref{eq def cap}).
Letting $j\rightarrow\infty$, we get
\begin{center}
	$Cap(\{v\leq P(u, f, \Omega)-2\epsilon\}, \Omega)\leq \lim\limits_{j\to\infty}Cap(\{u_j\leq P(u, f, \Omega)-2\epsilon\}, \Omega)=0.$
\end{center}
Then $v\geq  P(u, f, \Omega)-2\epsilon$. Since $\epsilon>0$ is arbitrary, we obtain $v\geq P(u, f, \Omega)$.
\end{proof}
\begin{Prop}\label{prop2 main1}
	Assume that $\Omega\subset\C^n$ is a bounded  pseudoconvex domain and
	$W\Subset\Omega$ is a smooth strictly pseudoconvex domain. Suppose that
	$0\leq f, g\in C(\Omega)$.
	If $u$ is a bounded viscosity 
	subsolution to the equation 
	\begin{equation}
	(dd^cw)^n=gd\lambda,
	\end{equation}
	on $\Omega$ then $(dd^cP(u, f, W))^n\leq \max\{f, g\}d\lambda$ in the pluripotential sense in $W$.
\end{Prop}
\begin{proof}
	Let $A>\sup_{\Omega}|u|$ and denote $B=\inf\{d(x, y): x\in W, y\in\partial\Omega\}>0$.
	 For every $m>m_0:=\frac{2A}{B}$, we consider
	\begin{center}
		$u_m(z):=\inf\{u(z+\xi)+m|\xi|: |\xi|<\dfrac{m_0B}{m} \},$
	\end{center}
for $z\in \overline{W}$.
Then $(u_m)$ is an increasing sequence of continuous functions in $\overline{W}$ satisfying
$\lim_{m\to\infty}u_m=u$ and
\begin{center}
	$(dd^cu_m)^n\leq g_md\lambda,$
\end{center}
in $W$ in the viscosity sense, where $g_m(z)=\sup\{g(z+\xi): |\xi|<\dfrac{m_0B}{m} \}.$ By Proposition \ref{prop1 main1}, we have
\begin{center}
	$(dd^cP(u_m, f, W))^n\leq \max\{f, g_m\}d\lambda,$
\end{center}
in the pluripotential sense in $W$ for all $m>m_0$.

Letting $m\rightarrow\infty$ and using Proposition \ref{prop3}, we obtain
\begin{center}
	$(dd^cP(u, f, W))^n\leq \max\{f, g\}d\lambda.$
\end{center}
\end{proof}
\begin{proof}[Proof of Theorem \ref{main1}]
	Since $\Omega$ is pseudoconvex, there exists an increasing sequence of 
	smooth strictly pseudoconvex domains $\Omega_j\Subset\Omega$ such that $\cup_j\Omega_j=\Omega$
	(see, for example, \cite[Theorem 2.6.11]{Hor73}).
	Let $f_k$ be a sequence of continuous functions in $\C^n$ such that $f_k$ converges to $f$ in $L^p$
	as $k\rightarrow\infty$. By Proposition \ref{prop2 main1}, we have
	\begin{center}
		$(dd^cP(u, f_k, \Omega_j))^n\leq \max\{f_k, g\}d\lambda,$
	\end{center}
in the pluripotential sense in $\Omega_j$ for all $j, k\in\Z^+$. Letting $j\rightarrow\infty$ and using
Proposition \ref{prop2}, we get
\begin{center}
		$(dd^cP(u, f_k, \Omega))^n\leq \max\{f_k, g\}d\lambda,$
\end{center}
in the pluripotential sense in $\Omega$ for all $j, k\in\Z^+$. Moreover, it follows from Proposition
\ref{prop4} that $P(u, f_k, \Omega)$ converges uniformly to $P(u, f, \Omega)$ as $k\rightarrow\infty$. Thus
\begin{center}
	$(dd^cP(u, f, \Omega))^n=\lim\limits_{k\to\infty}(dd^cP(u, f_k, \Omega))^n
	\leq \lim\limits_{k\to\infty}\max\{f_k, g\}d\lambda=\max\{f, g\}d\lambda.$
\end{center}
The proof is completed.
\end{proof}
\subsection{Proof of Theorem \ref{main0}}
We proceed through some lemmas.
\begin{Lem}\label{lem1 main0}
	Assume that $\Omega$ is a smooth strictly pseudoconvex domain and $u, f\in C^{\infty}(\overline{\Omega})$
	with $f\geq 0$.  Then, there exists $C>0$ such that, for every $\delta>0$, if 
	$\Omega_{\delta}:=\{z\in\Omega: d(z, \partial\Omega)>\delta \}\neq\emptyset$ then
	\begin{center}
		$P(u, f, \Omega_{\delta})\leq P(u, f, \Omega)+C\delta,$
	\end{center}
	on $\Omega_{\delta}$.	
\end{Lem}
\begin{proof}
	Since  $\Omega$ is a smooth strictly pseudoconvex, there exists
	$\rho\in C^{\infty}(\overline{\Omega})\cap PSH(\Omega)$ such that $\rho|_{\partial\Omega}0$,
	$\inf_{\Omega}\det(\rho_{\alpha\overline{\beta}})>0$  and
	\begin{center}
		$-C_1 d(z, \partial\Omega)\leq \rho(z)\leq -C_2d(z, \partial\Omega),$
	\end{center}
	for every $z\in\Omega$, where $0<C_2<C_1$.
	
	Let $M\gg 1$ such that $(M\rho+u)$ is plurisubharmonic in $\Omega$ and $(dd^c(M\rho+u))^n\geq fd\lambda$.
	
	For every $0<\delta<1$, if $\Omega_{\delta}\neq\emptyset$ then we define
	\begin{center}
		$v_{\delta}=
		\begin{cases}
		M\rho+u\quad\mbox{on}\quad\Omega\setminus \Omega_{\delta},\\
		\max\{M\rho+u, P(u, f, \Omega_{\delta})-2MC_1\delta\}\quad\mbox{on}\quad \Omega_{\delta}.\\
		\end{cases}$
	\end{center}
	Then $v_{\delta}\in PSH(\Omega)\cap L^{\infty}(\Omega)$, $v_{\delta}\leq u$ 
	and $(dd^cv_{\delta})^n\geq fd\lambda$. Hence 
	\begin{equation}\label{eq1 lem1 main0}
	v\leq P(u, f, \Omega).
	\end{equation}
	Moreover, by the definition of $v_{\delta}$, we have
	\begin{equation}\label{eq2 lem1 main0}
	v_{\delta}|_{\Omega_{\delta}}\geq P(u, f, \Omega_{\delta})-2MC_1\delta.
	\end{equation}
	By \eqref{eq1 lem1 main0} and \eqref{eq2 lem1 main0}, we obtain
	\begin{center}
		$P(u, f, \Omega_{\delta})\leq P(u, f, \Omega)+2MC_1\delta,$
	\end{center}
	on $\Omega_{\delta}$.
	
	The proof is completed.
\end{proof}
\begin{Lem}\label{lem2 main0}
	Let $u\in PSH(\Omega)\cap L^{\infty}(\Omega)$ and $0\leq f\in L^{p}(\Omega), p>1$. Suppose that
	$\phi: \Omega\rightarrow W$ is a biholomorphic mapping. Then
	\begin{center}
		$P(u, f, \Omega)\circ \phi^{-1}=P(u\circ\phi, (f\circ\phi).|J_{\phi}|^2, \phi^{-1}(\Omega)).$
	\end{center}
\end{Lem}
\begin{proof}
	The proof is straightforward from the definitions of $P(u, f, \Omega)$ and 
	$P(u\circ\phi, (f\circ\phi).|J_{\phi}|^2, \phi^{-1}(\Omega)).$
\end{proof}
\begin{Lem}\label{lem3 main0}
	Assume that $\Omega$ is a smooth strictly pseudoconvex domain and $u, f\in C^{\infty}(\overline{\Omega})$
	with $f\geq 0$. Then $ P(u, f, \Omega)$ is Lipschitz.
\end{Lem}
\begin{proof}
	Since $\Omega$ is bounded and smooth, there exists a constant $A>0$ such that, for every $z_0, z\in\Omega$,
	\begin{center}
		$|z-z_0|\leq A\inf\{length(\gamma): \gamma\in C^1([0,1], \Omega), \gamma (0)=z_0, \gamma(1)=z \}$.
	\end{center}
	Hence, $ P(u, f, \Omega)$ is Lipschitz iff
	\begin{equation}\label{eq0 lem3 main0}
	\sup\limits_{z_0\in\Omega}\limsup\limits_{z\to z_0}
	\dfrac{|P(u, f, \Omega)(z)-P(u, f, \Omega)(z_0)|}{|z-z_0|}<\infty.
	\end{equation}
	Let $a,b\in\Omega, a\neq b,$ such that 
	$\delta:=|a-b|\leq \dfrac{1}{2}\min\{d(a, \partial\Omega), d(b, \partial\Omega) \}$. 
	Since $a-b+\Omega_{\delta}\subset\Omega$, we have
	\begin{equation}\label{eq1 lem3 main0}
	P(u, f, \Omega)(a)\leq P(u, f, b-a+\Omega_{\delta})(a)
	\end{equation}
	
	By using Lemma \ref{lem2 main0} for $\phi(z)=z+b-a$, we have
	\begin{equation}\label{eq2 lem3 main0}
	P(u, f, b-a+\Omega_{\delta})(a)=P(u(z+b-a), f(z+b-a), \Omega_{\delta})(b).
	\end{equation}
	Since $u, f\in C^1(\overline{\Omega})$, there exists $C_1>0$ such that
	\begin{center}
		$|u(z)-u(z')|\leq C_1|z-z'|,$
	\end{center}
	and
	\begin{center}
		$|f(z)-f(z')|\leq C_1|z-z'|,$
	\end{center}
	for every $z, z'\in\Omega$. Hence
	\begin{equation}\label{eq3 lem3 main0}
	P(u(z+b-a), f(z+b-a), \Omega_{\delta})(b)
	\leq C_1\delta+P(u, (f-C_1\delta)_+, \Omega_{\delta})(b).
	\end{equation}
	By Proposition \ref{prop4}, there exists $C_2>0$ that does not depend on $\delta$ such that
	\begin{equation}\label{eq4 lem3 main0}
	P(u, (f-C_1\delta)_+, \Omega_{\delta})\leq C_2\delta+P(u, f, \Omega_{\delta}).
	\end{equation}
	Moreover, it follows from Lemma \ref{lem1 main0} that
	\begin{equation}\label{eq5 lem3 main0}
	P(u, f, \Omega_{\delta})\leq C_3\delta+P(u, f, \Omega),
	\end{equation}
	where $C_3>0$ that does not depend on $\delta$. By combining \eqref{eq1 lem3 main0},
	\eqref{eq2 lem3 main0}, \eqref{eq3 lem3 main0}, \eqref{eq4 lem3 main0}
	and \eqref{eq5 lem3 main0}, we obtain
	\begin{center}
		$P(u, f, \Omega)(a)\leq P(u, f, \Omega)(b)+C\delta,$
	\end{center}
	where $C>0$ that does not depend on $\delta$. Similarly, we have
	\begin{center}
		$P(u, f, \Omega)(b)\leq P(u, f, \Omega)(a)+C\delta.$
	\end{center}
	Then
	\begin{center}
		$|P(u, f, \Omega)(a)-P(u, f, \Omega)(b)|\leq C\delta=C|a-b|.$
	\end{center}
	Thus, for every $a\in\Omega$,
	\begin{center}
		$\limsup\limits_{b\to a}
		\dfrac{|P(u, f, \Omega)(b)-P(u, f, \Omega)(a)|}{|b-a|}\leq C,$
	\end{center}
	and we get \eqref{eq0 lem3 main0}.
\end{proof}
\begin{proof}[Proof of Theorem \ref{main0}]
	Let $u_j, f_j\in C^{\infty}(\overline{\Omega})$ such that $f_j\geq 0$, $u_j$ converges uniformly
	to $u$ and $f_j$ converges in $L^p(\Omega)$ to $f$. By Lemma \ref{lem3 main0}, we have
	$P(u_k, f_j, \Omega)$ is continuous for every $j, k\in\Z^+$. Moreover, since $u_j$ converges uniformly
	to $u$, we have $P(u_k, f_j, \Omega)$ converges uniformly
	to $P(u, f_j, \Omega)$ as $k\rightarrow\infty$. Hence $P(u, f_j, \Omega)$ is continuous for every $j$.
	Since $f_j$ converges in $L^p(\Omega)$ to $f$, it follows from Proposition \ref{prop4} that
	$P(u, f_j, \Omega)$ converges uniformly
	to $P(u, f, \Omega)$ as $j\rightarrow\infty$. Then $P(u, f, \Omega)$ is continuous.
	
	The proof is completed.
\end{proof}
\section*{Acknowledge}
This paper was partially written while the second author
visited Vietnam Institute for Advanced Study in Mathematics(VIASM). She would like to thank  
the institution for its hospitality.

\end{document}